\documentclass[11pt,english,oneside]{amsart}

\usepackage[breaklinks,colorlinks,linkcolor=blue,citecolor=blue,urlcolor=blue]{hyperref}
\usepackage[T1]{fontenc}
\allowdisplaybreaks[4]
\usepackage[latin9]{inputenc}
\usepackage{geometry}
\usepackage{indentfirst}
\usepackage{amssymb}
\usepackage{color}
\usepackage{graphicx}
\usepackage{subfigure}
\usepackage{enumerate}
\usepackage{float}

\usepackage{booktabs}
\usepackage{array, caption, threeparttable}
\usepackage[font=small,labelfont=bf,labelsep=none]{caption}
\usepackage{graphicx,color}
\usepackage{amsmath, amssymb, graphics}
\usepackage{graphicx}

\numberwithin{figure}{section}

\makeatletter
\numberwithin{equation}{section} 
\numberwithin{figure}{section} 
\@ifundefined{theoremstyle}{\usepackage{amsthm}}{}
\theoremstyle{plain}

\theoremstyle{plain}

\theoremstyle{plain}

\theoremstyle{remark}

\theoremstyle{remark}

\theoremstyle{plain}

\usepackage{geometry}

\def\exp{\hbox{\rm exp}}
\smallskip
\def\<{{\langle }}
\def\>{{\rangle }}

\makeatother

\usepackage{babel}

\def\exp{\hbox{\rm exp}}
\smallskip
\def\<{{\langle }}
\def\>{{\rangle }}

\makeatother

\usepackage[mathscr]{eucal}
\usepackage{amssymb}
\usepackage{latexsym}
\usepackage{amsthm}
\theoremstyle{plain}
\newtheorem{theorem}{Theorem}[section]
\newtheorem*{theorem*}{Theorem}
\newtheorem*{thm*}{Theorem}

\newtheorem{corollary}{Corollary}[section]
\newtheorem{remark}{Remark}[section]
\newtheorem{lemma}{Lemma}[section]

\newtheorem{definition}{Definition}[section]

\makeatletter
\@addtoreset{equation}{section}

\usepackage{color}

\usepackage{graphicx,color}

\title[$\delta$-stable minimal hypersurfaces ]{Complete two-sided $\delta$-stable minimal hypersurfaces in $\mathbf R^{n+1}$}
\author{Qing-Ming Cheng and  Guoxin Wei}
\address{Qing-Ming Cheng \newline \indent Mathematical Science Research Center, \newline
\indent Chongqing University of Technology, Chongqing 400054,  P. R. China  \newline \indent
qingmingcheng@yahoo.com, chengqingming@cqut.edu.cn}
\address{Guoxin Wei \newline \indent School of Mathematical Sciences, \newline \indent South China Normal University,
510631, Guangzhou,  China \newline \indent weiguoxin@tsinghua.org.cn}

\begin{document}
\maketitle

\begin{abstract} In this paper, we study complete $\delta$-stable minimal hypersurfaces in $\mathbf R^{n+1}$. We prove that  complete two-sided
$\delta$-stable minimal hypersurfaces have Euclidean volume growth if  $3\leq n\leq 5$ and $\delta>\delta_0(n)$, where
$\delta_0(3)=1/3$, $\delta_0(4)=1/2$  and  $\delta_0(5)=21/22$. We also give a sufficient  condition such that complete two-sided $\delta$-stable minimal
hypersurfaces in $\mathbf R^{n+1}$ is the hyperplane. Furthermore, we prove that a complete two-sided
$\delta$-stable minimal hypersurface is  the hyperplane if  $3\leq n\leq 5$ and $\delta>\delta_1(n)$, where
$\delta_1(3)=3/8$, $\delta_1(4)=2/3$  and  $\delta_1(5)=21/22$.
\end{abstract}
\maketitle
\section{ Introduction}
\noindent
Let $X:M^n\to \mathbf R^{n+1}$ be a hypersurface in the Euclidean space $\mathbf R^{n+1}$.
If the mean curvature $H=0$ of $X:M^n\to \mathbf R^{n+1}$,  $X:M^n\to \mathbf R^{n+1}$ is
called a minimal hypersurface. It is well-known that minimal hypersurfaces in $ \mathbf R^{n+1}$
are critical points of the $n$-volume functional.  A hypersurface in $\mathbf R^{n+1}$ is called
a graph if $X:M^n\to \mathbf R^{n+1}$  can be expressed  by
$$
X(x_1, x_2, \cdots, x_n)=(x_1, x_2, \cdots, x_n, u(x_1, \cdots, x_n)), \ \ (x_1, x_2, \cdots, x_n) \in M^n \subset \mathbf R^n,
$$
where $u=u(x_1, x_2, \cdots, x_n)$ is a smooth function. If $M^n = \mathbf R^n$, the graph $X:M^n\to \mathbf R^{n+1}$
is called an entire graph over  $\mathbf R^n$.

\noindent
It is also well-known that Bernstein \cite{b1}  proved that an entire minimal graph in $\mathbf R^3$ is the plane $\mathbf R^2$. The same
problem for higher dimensions, which is called the Bernstein problem, was studied by Fleming \cite{f}, De Giorgi \cite{d},  Almgren \cite{a}
and Simons \cite{s}.
They showed that an entire minimal graph in $\mathbf R^{n+1}$ is the  plane $\mathbf R^n$ for $n\leq 7$. For $n\geq 8$,
Bombieri, De Giorgi and Giusti \cite{bgg} were able to construct entire minimal graphs that are not hyperplane.
Hence, the so-called Bernstein problem was resolved completely. \\
Since a minimal hypersurface $X:M^n\to \mathbf R^{n+1}$ is a critical point of the $n$-volume functional, the second variation of the
$n$-volume functional is given by
$$
\int_M\big (|\nabla \varphi|^2-S\varphi^2\big)dv, \ \varphi \in \mathcal C_c^{1}(M),
$$
where $S=|A_M|^2$ and $A_M$ denotes the  second fundamental form of hypersurface $X:M^n\to \mathbf R^{n+1}$.
A minimal hypersurface $X:M^n\to \mathbf R^{n+1}$ is called stable if, for any $\varphi \in \mathcal C_c^{1}(M)$,
$$
\int_M\big (|\nabla \varphi|^2-S\varphi^2\big)dv\geq 0
$$
holds.
Minimal graphs are stable. As a natural generalization of the Bernstein problem,
one asks whether an $n$-dimensional complete two-sided stable minimal hypersurface $X:M^n\to \mathbf R^{n+1}$
in $\mathbf R^{n+1}$ is a hyperplane? This problem is called the stable Bernstein problem.
For the stable Bernstein problem, do Carmo and Peng \cite{cp}, Fischer-Colbrie and Schoen \cite{fs} and Pogorelov \cite{p}
resolved it for $n=2$, affirmatively,  that is, they proved that the plane is the only complete two-sided
stable minimal surfaces in $\mathbf R^3$.\\
For higher dimensions, Schoen, Simon and Yau \cite{ssy} made an important breakthrough. They showed that the hyperplane
$\mathbf R^n$ is the only complete two-sided stable minimal hypersurfaces with Euclidean volume growth
 in $\mathbf R^{n+1}$ for $3\leq n\leq 5$. For $n=6$, Schoen and Simon \cite{ss} also gave
a positive answer under an additional condition that $X:M^n\to \mathbf R^{n+1}$  is embedded.
Very recently, Bellettini \cite{b} has obtained the same result for $n=6$ as one of  Schoen, Simon and Yau \cite{ssy} under the
Euclidean volume growth condition
$$
 \text{\rm vol }(M\cap X^{-}(B_R^{n+1}) \leq \Lambda R^n.
 $$
We should notice that Schoen, Simon and Yau \cite{ssy} used volume growth condition of the intrinsic geodesic balls.
Bellettini \cite{b} made use of the extrinsic volume growth condition of balls.  Since,  for $n= 7$,
Bombieri, De Giorgi and Giusti \cite{bgg} were able to construct complete two-sided stable minimal hypersurfaces which are not flat,
therefore, in order to resolve the stable Bernstein problem, one needs to   prove that
complete two-sided  stable minimal hypersurfaces in $\mathbf R^{n+1}$ for $3\leq n\leq 6$ have Euclidean volume growth.
Up until very recently, Chodosh and Li \cite{cl2, cl3, cls, clms}  have made very important contributions. Namely, Chodosh and Li \cite{cl3}
have resolved the stable Bernstein problem for $n=3$. Later, very dramatically, Chodosh and Li \cite{cl2}  have found  the second strategy
to resolve the stable Bernstein problem for $n=3$ and  Catino, Mastrolia and Roncoroni \cite{cmr} have developed a new
 strategy to  resolved the stable Bernstein problem for $n=3$ at the almost same time by the technique of conformal transformations.
By making use of ideas of the second strategy of Chodosh and Li \cite{cl3} and bi-Ricci curvature, Chodosh, Li, Minter and Stryker
\cite{clms} have been able  to resolve the stable Bernstein problem for $n=4$ and Mazet \cite{m}, by  following the second strategy of
Chodosh and Li \cite{cl3} and using the weighted bi-Ricci curvature, has been able to resolve the stable Bernstein problem for $n=5$.
But,  the stable Bernstein problem for the last case $n=6$ is still open.

\noindent
As a natural generalization of the concept of stable, in order to study the total curvature of embedded minimal disks,
Colding and Minicozzi \cite{cm} introduced the concept of $\delta$-stable as following:

\begin{definition}
An $n$-dimensional  minimal hypersurface $X:M^n\to \mathbf R^{n+1}$
in $\mathbf R^{n+1}$  is called $\delta$-stable, $\delta>0$ if
$$
\int_M\big (|\nabla \varphi|^2-\delta S\varphi^2\big)dv\geq 0, \ \varphi \in \mathcal C_c^{1}(M)
$$
holds.
\end{definition}
\begin{remark}
It is easy to know that stable is $\delta$-stable and $1$-stable is stable.
\end{remark}
\noindent
$\delta$-stable minimal hypersurfaces in $\mathbf R^{n+1}$  are closely related to anisotropic minimal
hypersurfaces in $\mathbf R^{n+1}$. We suggest readers to see the  reference \cite{cl3} for details.
On  $\delta$-stable minimal hyperusfraces in $\mathbf R^{n+1}$, Tam and Zhou \cite{tz} proved that
$n$-dimensional catenoid is $\frac{n-2}n$-stable. Furthermore, they proved that an $n$-dimensional
complete two-sided $\frac{n-2}n$-stable minimal is a hyperplane or Catenoid if
$$
\lim_{R\to\infty}\dfrac{\int_{B_{2R}(p_0)\setminus B_R(p_0)}S^{\frac{n-2}n}dv}{R^2}=0.
$$
where $B_R(p_0)$ denotes the geodesic ball with radius  $R$ and centered at $p_0$.
Anderson \cite{a}, Fu and Li \cite{fl},
and so on proved that an $n$-dimensional complete two-sided $\delta$-stable minimal hypersurface
$X:M^n\to \mathbf R^{n+1}$ with finite total curvature is a hyperplane if $\delta>\frac{n-2}n$ and is either
a hyperplane or a catenoid if
$\delta=\frac{n-2}{n}$.

\noindent
 For general case, Kawai \cite{k} proved that a  complete two-sided $\delta$-stable minimal surface
$X:M^2\to \mathbf R^{3}$ is a plane if $\delta>\frac18$.    Cheng and Zhou \cite{cz}
have proved that an $n$-dimensional complete two-sided $\frac{n-2}n$-stable minimal hypersurface
$X:M^n\to \mathbf R^{n+1}$ has either one end or is   a catenoid if
$$
\begin{cases}
\lim_{R\to\infty}\dfrac{\sup_{B(R)} \sqrt S}{R^{\frac{n-3}2}}=0, &  n>3,\\
\lim_{R\to\infty}\dfrac{\sup_{B(R)} \sqrt S}{\log R}=0, &  n=3.
 \end{cases}
 $$
 Very recently, Hong, Li and Wang \cite{hlw} have studied $\delta$-stable minimal hypersurfaces. They have proved
 that for $n\geq 3$ and $\delta>\max\{\frac{n-2}n, \frac{(n-2)^2}{4(n-1)}\}$, an $n$-dimensional complete
 two-sided $\delta$-stable minimal hypersurface in $\mathbf R^{n+1}$
 satisfying the Euclidean volume growth condition
 $$
 \text{\rm vol }(M\cap X^-(B_R^{n+1}) \leq \Lambda R^n
 $$
is flat.
In this paper, we prove
\begin{theorem}
For an $n$-dimensional complete
 two-sided $\delta$-stable minimal hypersurface $X:M^n\to \mathbf R^{n+1}$  in $\mathbf R^{n+1}$
 with  $\delta>\frac{n-2}n$, if, for some $q$ with $\frac{n-2}{n}<q<\delta$ such that
 $$
 \dfrac{\int_{B_{R} (p_0)}S^{\frac{qn}{n-2}}dv}{R^{\frac{(n-2)}{q}-2}}\leq \varepsilon_1,
 $$
 holds for sufficient large  $R>1$,  $X:M^n\to \mathbf R^{n+1}$ is a hyperplane,
 where   $B_R(p_0)$ is a geodesic ball of radius $R$ centered at some point $p_0$ in $M^n$
 and $\varepsilon_1$ is a small constant depending on the dimension $n$, $\delta$ and $q$.
\end{theorem}
\begin{theorem}\label{thm:1.2}
For an $n$-dimensional complete two-sided $\delta$-stable minimal hypersurface $X:M^n\to \mathbf R^{n+1}$
in $\mathbf R^{n+1}$ with $\delta>\frac{n-2}n$, we have  the following inequality:
\begin{equation}
\begin{aligned}
&\int_{M}S^{2k+1}f^2dv
\leq C_1\int_MS^{2k}|\nabla f |^2dv, \  \ f\in \mathcal C_c^1(M), \\
&\int_{M}(\sqrt Sf)^pdv
\leq C_2\int_M|\nabla f |^pdv, \  \ f\in \mathcal C_c^1(M) \\
\end{aligned}
\end{equation}
 with $p=4k+2$,
 where  $k$ satisfies $\delta-\sqrt{\delta(\delta-\frac{n-2}n)}<2k<\delta+\sqrt{\delta(\delta-\frac{n-2}n)}$
 and $C_1$  and $C_2$ are  positive constants  depending on $n, \ \delta, \ k$.
 \end{theorem}
\begin{corollary}
For an $n$-dimensional complete
 two-sided $\delta$-stable minimal hypersurface $X:M^n\to \mathbf R^{n+1}$  in $\mathbf R^{n+1}$
 with  $\delta>\frac{n(n-2)}{4(n-1)}$,  if
 $$
 \text{\rm vol }\{B_{R} (p_0)\} \leq \Lambda R^n,
 $$
 holds for sufficient large  $R>1$,  $X:M^n\to \mathbf R^{n+1}$ is a hyperplane,
 where   $B_R(p_0)$ is a geodesic ball of radius $R$ centered at some point $p_0$ in $M^n$.
 \end{corollary}
 \begin{proof} According to the theorem \ref{thm:1.2}, for $k$ satisfying
 $\delta-\sqrt{\delta(\delta-\frac{n-2}n)}<2k<\delta+\sqrt{\delta(\delta-\frac{n-2}n)}$,
 we have with $p=4k+2$
 \begin{equation*}
\begin{aligned}
&\int_{M}(\sqrt Sf)^pdv
\leq C_2\int_M|\nabla f |^pdv, \  \ f\in \mathcal C_c^1(M). \\
\end{aligned}
\end{equation*}
Taking  $f=1$ in  $B_R(p_0)$ and $f=0$ in $M \setminus B_{2R}(p_0)$ and $0\leq f\leq 1$ in $M$
and $|\nabla f|\leq \frac2R$,  we infer
\begin{equation}\label{eq:1.3}
\begin{aligned}
&\int_{B_R(p_0)}(\sqrt S)^pdv
\leq C_2\int_{B_{2R}(p_0)}(\dfrac{2}{R})^pdv\leq C_22^p\Lambda R^{n-p}. \\
\end{aligned}
\end{equation}
Since $\delta>\frac{n(n-2)}{4(n-1)}$, for this fixed $\delta$, there exists a sufficiently small $\epsilon$
such that  $\delta>\frac{n(n-2)}{4(n-1)}+\epsilon$.
We take $2k=\frac{n(n-2)}{4(n-1)}+\epsilon+\sqrt{(\frac{n(n-2)}{4(n-1)}+\epsilon)(\frac{n(n-2)}{4(n-1)}+\epsilon-\frac{n-2}n)}$. We derive
$$
\begin{aligned}
&2k>\frac{n(n-2)}{4(n-1)}+\sqrt{\frac{n(n-2)}{4(n-1)}(\frac{n(n-2)}{4(n-1)}-\frac{n-2}n)}\\
&=\frac{(n-2)}2\\
\end{aligned}
$$
Hence, we know $p=4k+2>n$. From (\ref{eq:1.3}), we know $S\equiv0$ since $R$ is arbitrary, that is,
$X:M^n\to \mathbf R^{n+1}$ is a hyperplane.
 \end{proof}

 \noindent
 For $n=3$, by making use of the conformal transformation,
 Catino,  Mari, Mastrolia and Roncoroni \cite{cmmr} have  proved that a $3$-dimensional
 complete two-sided $\frac13$-stable minimal hypersurface in $\mathbf R^4$ has one end or  is a catenoid.
 By making  use of the  method of the moving frame, introducing the ${\rm Bi}_{(\alpha, \beta)}{\rm Ric}$ curvature
 and essentially in the frame of  the strategy of Chodoch and Li \cite{cl3}, we prove the following
 \begin{theorem}
 For $3\leq n\leq 5$ and $\delta>\delta_1(n)$, an $n$-dimensional complete
 two-sided $\delta$-stable minimal hypersurface $X:M^n\to \mathbf R^{n+1}$  in $\mathbf R^{n+1}$ is a hyperplane,
 where $\delta_1(3)=3/8$, $\delta_1(4)=2/3$ and
 $\delta_1(5)=21/22$.
 \end{theorem}
\begin{remark}  For $n=3, \ 4$ the results in the above theorem has been proved by Hong, Li and Wang in \cite{hlw}.
Furthermore, since stable is $\delta$-stable,  we reproved the stable Bernstein problem for $3\leq n\leq 5$.
\end{remark}

\noindent
{\bf Acknowlogement}. Authors would like to thank professor Li Haizhong and professors Hong Han and Wang Gaoming
for helpful discussions.

 \vskip3mm

 \vskip3mm

\section{Preliminary}
\vskip2mm
\noindent
Let $X:M^n\to \mathbf R^{n+1}$ be an $n$-dimensional
hypersurface in $\mathbf R^{n+1}$. By a parallel translation,   we can assume
$X(p)=O$ for some $p\in M$. We choose a local orthonormal frame
$\{\vec{e}_1, \cdots, \vec{e}_{n}, \vec{e}_{n+1}\}$ and the dual
coframe $\{\omega_1, \cdots,$ $\omega_n$, $ \omega_{n+1}\}$ in
such  that $\{\vec{e}_1, \cdots, \vec{e}_n\}$  is a local
orthonormal frame on $M^n$. Thus, the induced metric $g$ of
 $X:M^n\to \mathbf R^{n+1}$ is given by $g=\sum_{i=1}^n\omega_i^2$.
Hence, we have
$$
\omega_{n+1}=0
$$
in $M^n$. According to Cartan lemma, one has
$$
\omega_{i,n+1}=\sum_{j}h_{ij}\omega_j, \ h_{ij}=h_{ji}.
$$
The mean curvature $H$ and the second fundamental form
$A_M$ of $M^n$ are defined, respectively, by
$$
H=\sum_{i}h_{ii}, \
A_M=\sum_{i,j}h_{ij}\omega_i\otimes\omega_j\vec{e}_{n+1}.
$$
If   $H\equiv 0$,
 $X:M^n\to \mathbf R^{n+1}$ is  called a  minimal hypersurface.
From the structure equations of $M^n$,  we have Gauss equations,
and Codazzi equations.
\begin{align*}
R_{ijkl}=(h_{ik}h_{jl}-h_{il}h_{jk}),
\end{align*}
\begin{align*}
h_{ijk}=h_{ikj},
\end{align*}
where $h_{ijk}$ are defined by
\begin{equation}
\sum_{k}h_{ijk}\omega_k=dh_{ij}+\sum_k
h_{ik}\omega_{kj} +\sum_k h_{kj}\omega_{ki}.
\end{equation}
Defining $h_{ijkl}$, $h_{ijklm}$  by
\begin{equation}
\sum_{l}h_{ijkl}\omega_l=dh_{ijk}+\sum_l
h_{ijl}\omega_{lk} +\sum_l h_{ilk}\omega_{lj}+\sum_m h_{ljk}\omega_{li},
\end{equation}
\begin{equation}
\begin{aligned}
&\sum_{m}h_{ijklm}\omega_m\\&=dh_{ijkl}+\sum_m
h_{ijkm}\omega_{ml} +\sum_m h_{ijml}\omega_{mk}+\sum_m h_{imkl}\omega_{mj} +\sum_m h_{mjkl}\omega_{mi},
\end{aligned}
\end{equation}
we have
\begin{align*}
h_{ijkl}-h_{ijlk}=\sum_ph_{im}R_{mjkl}+\sum_mh_{mj}R_{mikl},
\end{align*}
\begin{align*}
h_{ijklm}-h_{ijkml}=\sum_ph_{ijp}R_{pklm}+\sum_ph_{ipk}R_{pjlm}+\sum_ph_{pjk}R_{pilm}.
\end{align*}
Let $r$ denote the distance function from the origin $O$ to $X$, $r=|X|$. We consider
Gulliver-Lawson conformal metric $\tilde g=r^{-2}g$. We have
$\tilde g=\sum_{i=1}^n\tilde \omega_i^2$ with $\tilde \omega_i=r^{-1}\omega_i$.
For hypersurface $(N, \tilde g)$ with $N=M\setminus X^{-1}(O)$, we have  connection $\tilde \omega_{ij}$
with respect to  $\tilde g$
$$
\tilde \omega_{ij}=\omega_{ij}+r_j\tilde \omega_i-r_i\tilde \omega_j
$$
with $dr=\sum_ir_i\omega_i$. Thus, we obtain the curvature tensor $\tilde R_{ijkl}$ with respect to $\tilde g$
\begin{equation}\label{eq2.4}
\begin{aligned}
\tilde R_{ijkl}&=r^2R_{ijkl}-|dr|^2(\delta_{ik}\delta_{jl}-\delta_{il}\delta_{jk})
-r(r_{jk}\delta_{il}-r_{ik}\delta_{jl}-r_{jl}\delta_{ik}+r_{il}\delta_{jk}),
\end{aligned}
\end{equation}
where $r_{ij}$ is defined  by $\sum_jr_{ij}\omega_j=dr_i+\sum_jr_j\omega_{ji}$.\\
From $r^2=\langle X, X\rangle$, we have
\begin{equation}
rr_i=\langle X, \vec e_i\rangle,  \  rr_{ij}=\delta_{ij}-r_ir_j+h_{ij}\langle X, \vec e_{n+1}\rangle.
\end{equation}
From  (\ref{eq2.4}), we infer
\begin{lemma}
\begin{equation}
\begin{aligned}
\tilde R_{ijkl}&=r^2R_{ijkl}+(2-|dr|^2)(\delta_{ik}\delta_{jl}-\delta_{il}\delta_{jk})\\
&-\langle X, \vec e_{n+1}\rangle (h_{jk}\delta_{il}-h_{ik}\delta_{jl}-h_{jl}\delta_{ik}+h_{il}\delta_{jk})\\
&+(r_{j}r_{k}\delta_{il}-r_{i}r_{k}\delta_{jl}-r_{j}r_{l}\delta_{ik}+r_{i}r_l\delta_{jk}).
\end{aligned}
\end{equation}
\end{lemma}
\noindent
We define  $(\alpha, \beta)$-bi-Ricci curvature ${\rm Bi_{(\alpha,\beta)}Ric}_{12}$ by
\begin{equation}
{\rm Bi_{(\alpha,\beta)}Ric}_{12}=\beta\sum_{i=1}^nR_{1i1i}+\alpha\sum_{j=3}^nR_{2j2j}.
\end{equation}
According to the lemma 2.1, we obtain the relation
between ${\rm Bi}_{(\alpha,\beta)}{\rm Ric}$ and $\widetilde {{\rm Bi}_{(\alpha,\beta)}{\rm Ric}}$
with respect to the metric $g$ and the metric $\tilde g$, respectively.
\begin{lemma}
\begin{equation}
\begin{aligned}
&\widetilde{\rm Bi_{(\alpha,\beta)}Ric}_{12}=r^2{\rm Bi_{(\alpha,\beta)}Ric}_{12}+2(n-1)\beta+2(n-2)\alpha\\
&-(n\beta+(n-1)\alpha)|dr|^2-((n-2)\beta-\alpha)r_1^2-(n-3)\alpha r_2^2\\
&+\langle X, \vec e_{n+1}\rangle \bigl [ ((n-2)\beta-\alpha)h_{11}+(n-3)\alpha h_{22}\bigl ].
\end{aligned}
\end{equation}
\end{lemma}

\section{ Key estimates}
\vskip3mm

\begin{lemma} \label{lem3.1} For $\frac{n-1}{n-2}a> \max\{\beta, \alpha\}$ and
$\frac{4n}{n-2}a^2-4(\frac{n-1}{n-2}\beta+\alpha)a +(4\beta-\alpha)\alpha>0$,
$$
\begin{aligned}
f(x, y)&=a\bigl [x^2+y^2+\dfrac{1}{n-2}(x+y)^2\bigl]-\beta x^2-\alpha(xy+y^2)\\
&-E[((n-2)\beta-\alpha)x+(n-3)\alpha y]
\end{aligned}
$$
satisfies
$$
\begin{aligned}
&f(x,y)\geq f_{min}(a, \alpha,\beta)\\
&=E^2\ \bigl\{ \dfrac{(n-2)\alpha^3-\bigl[(n^2-5n+8)a+(3n-7)\beta\bigl]\alpha^2}{\frac{4n}{n-2}a^2
-4(\frac{n-1}{n-2}\beta+\alpha)a +(4\beta-\alpha)\alpha}\\
&+\dfrac{\bigl[(n-2)^2\alpha-(n-1)(n-2)a\bigl]\beta^2+4(n-2)a\alpha\beta}
{\frac{4n}{n-2}a^2-4(\frac{n-1}{n-2}\beta+\alpha)a +(4\beta-\alpha)\alpha}\bigl\},
\end{aligned}
$$
where $E$ does not depend on $x, \ y$.
\end{lemma}
\begin{proof} From
$$
\begin{aligned}
&f_x=2ax+\dfrac{2a}{n-2}(x+y)-2\beta x-\alpha y-E((n-2)\beta-\alpha)=0,\\
&f_y=2ay+\dfrac{2a}{n-2}(x+y)-2\alpha y-\alpha x-E(n-3)\alpha=0,
\end{aligned}
$$
we have the critical point
$$
\begin{cases}
&x=E\ \dfrac{(n-1)\alpha^2-2((n-2)\beta +2a)\alpha+2(n-1)a\beta}{\frac{4n}{n-2}a^2-4(\frac{n-1}{n-2}\beta+\alpha)a +(4\beta-\alpha)\alpha}, \\ \\
&y=E\ \dfrac{\alpha^2+[(n-4)\beta-2(n-2)a]\alpha+2a\beta}{\frac{4n}{n-2}a^2-4(\frac{n-1}{n-2}\beta+\alpha)a +(4\beta-\alpha)\alpha}.
\end{cases}
$$
Since
$$
\begin{aligned}
&f_{xx}=\dfrac{2(n-1)}{n-2}a-2\beta>0,\\
&f_{yy}=\dfrac{2(n-1)}{n-2}a-2\alpha>0, \\
&f_{xy}=\dfrac{2a}{n-2}-\alpha,
\end{aligned}
$$
and $\frac{4n}{n-2}a^2-4(\frac{n-1}{n-2}\beta+\alpha)a +(4\beta-\alpha)\alpha>0$,
$f(x, y)$ attains  its minimum $f_{min}(a, \alpha,\beta)$ at the critical point
$$
\begin{aligned}
&f_{min}(a, \alpha,\beta)\\
&=E^2\ \bigl\{ \dfrac{(n-2)\alpha^3-\bigl[(n^2-5n+8)a+(3n-7)\beta\bigl]\alpha^2}{\frac{4n}{n-2}a^2
-4(\frac{n-1}{n-2}\beta+\alpha)a +(4\beta-\alpha)\alpha}\\
&+\dfrac{\bigl[(n-2)^2\alpha-(n-1)(n-2)a\bigl]\beta^2+4(n-2)a\alpha\beta}{\frac{4n}{n-2}a^2-4(\frac{n-1}{n-2}\beta+\alpha)a +(4\beta-\alpha)\alpha}\bigl\}.
\end{aligned}
$$
\end{proof}

\begin{lemma}\label{lem3.2} For an $n$-dimensional  minimal hypersurface $X:M^n\to \mathbf R^{n+1}$
in $\mathbf R^{n+1}$, we have
$$
a S\geq - {\rm Bi}_{(\alpha,\beta)}{\rm Ric}_{12}-\dfrac{\langle X,\vec e_{n+1}\rangle}{r^2} \big[((n-2)\beta-\alpha)h_{11}
+(n-3)\alpha h_{22}\big]+f_{\min}(a,\alpha,\beta)
$$
with $E=\dfrac{\langle X,\vec e_{n+1}\rangle}{r^2}$.
\end{lemma}
\begin{proof}
At each point $p$, we choose  a local orthonormal frame
$\{\vec{e}_1, \cdots, \vec{e}_{n}\}$  such that $h_{ij}=\lambda_i\delta_{ij}$.
We know
$$
\sum_i\lambda_i=0,  \ S=\sum_i\lambda_i^2, \ {\rm Bi}_{(\alpha,\beta)}{\rm Ric}_{12}
=-\beta\lambda_1^2-\alpha(\lambda_1\lambda_2+\lambda_2^2).
$$
Thus,
$$
\begin{aligned}
&a S+ {\rm Bi}_{(\alpha,\beta)}{\rm Ric}_{12}
+\dfrac{\langle X,\vec e_{n+1}\rangle}{r^2} \big[((n-2)\beta-\alpha)h_{11}+(n-3)\alpha h_{22}\big]\\
&\geq a\bigl[\lambda_1^2+\lambda_2^2+\dfrac{1}{n-2}(\lambda_1+\lambda_2)^2\bigl]
-\beta\lambda_1^2-\alpha(\lambda_1\lambda_2+\lambda_2^2)\\
&+\dfrac{\langle X,\vec e_{n+1}\rangle}{r^2} \big[((n-2)\beta-\alpha)\lambda_1+(n-3)\alpha \lambda_2\big].
\end{aligned}
$$
From the lemma \ref{lem3.1} and putting $E=-\dfrac{\langle X,\vec e_{n+1}\rangle}{r^2}$, we finish the  proof of the  lemma.
\end{proof}

\begin{lemma}\label{lem3.3} For $a=b\delta_0(n)$ with $\delta>\delta_0(n)$,
\begin{equation}
\begin{aligned}
&b\bigl(-\dfrac{n(n-2)}2+\dfrac{n^2-4}4|dr|^2\bigl)+r^2f_{\min}(a, \alpha,\beta)
+2(n-1)\beta+2(n-2)\alpha\\
&-(n\beta+(n-1)\alpha)|dr|^2-((n-2)\beta-\alpha)r_1^2-(n-3)\alpha r_2^2\\
&\geq \varepsilon(n)>0,\\
\end{aligned}
\end{equation}
where
\begin{equation}\label{eq:3.2}
\begin{aligned}
&\delta_0 (3)=\frac13, \ \delta_0 (4)=\frac12,  \  \ \delta_0(5)=\frac{21}{22},\\
&\varepsilon(3)=\frac9{11}, \ \varepsilon(4)=\frac{377}{5260},  \  \ \varepsilon(5)=\frac{979826999}{65363627000}\approx 0.014999.
\end{aligned}
\end{equation}
\end{lemma}
\begin{proof}
Since $r^2E^2=\frac{\langle X, \vec e_{n+1}\rangle ^2}{r^2}=1-|dr|^2$,  we consider
\begin{equation}
\begin{aligned}
&F(n, b,  \alpha,  \beta, |dr|^2):=2(n-1)\beta+2(n-2)\alpha-b\dfrac{n(n-2)}2\\
&+\bigl[\dfrac{n^2-4}4b-(n\beta+(n-1)\alpha)-\max\{((n-2)\beta-\alpha), (n-3)\alpha\}\bigl]|dr|^2\\
&+(1-|dr|^2) \bigl\{ \dfrac{(n-2)\alpha^3-\bigl[(n^2-5n+8)a+(3n-7)\beta\bigl]\alpha^2}{\frac{4n}{n-2}a^2
-4(\frac{n-1}{n-2}\beta+\alpha)a +(4\beta-\alpha)\alpha}\\
&+\dfrac{\bigl[(n-2)^2\alpha-(n-1)(n-2)a\bigl]\beta^2+4(n-2)a\alpha\beta}
{\frac{4n}{n-2}a^2-4(\frac{n-1}{n-2}\beta+\alpha)a +(4\beta-\alpha)\alpha}\bigl\}.\\
\end{aligned}
\end{equation}
Thus, we know that $F(n, b,  \alpha,  \beta, |dr|^2)$ is linear on $|dr|^2$ and
\begin{equation}
\begin{aligned}
&F(n, b, \alpha, \beta, 0)=2(n-1)\beta+2(n-2)\alpha-b\dfrac{n(n-2)}2\\
&+\bigl\{ \dfrac{(n-2)\alpha^3-\bigl[(n^2-5n+8)a+(3n-7)\beta\bigl]\alpha^2}{\frac{4n}{n-2}a^2
-4(\frac{n-1}{n-2}\beta+\alpha)a +(4\beta-\alpha)\alpha}\\
&+\dfrac{\bigl[(n-2)^2\alpha-(n-1)(n-2)a\bigl]\beta^2+4(n-2)a\alpha\beta}
{\frac{4n}{n-2}a^2-4(\frac{n-1}{n-2}\beta+\alpha)a +(4\beta-\alpha)\alpha}\bigl\},\\
\end{aligned}
\end{equation}
\begin{equation}
\begin{aligned}
&F(n, b, \alpha, \beta, 1)=2(n-1)\beta+2(n-2)\alpha-b\dfrac{n(n-2)}2\\
&+\bigl[\dfrac{n^2-4}4b-(n\beta+(n-1)\alpha)-\max\{((n-2)\beta-\alpha), (n-3)\alpha\}\bigl].\\
\end{aligned}
\end{equation}
Taking
\begin{equation}\label{eq:3.5}
\begin{cases}
&a= \dfrac{10}{11}, \ b=\dfrac{30}{11}, \ \alpha=\dfrac{18}{11}, \  \beta=\dfrac32, \  \text{\rm when}\  n=3,\\ \\
&a=\dfrac{24}{25}, \  b=\dfrac{48}{25},  \  \alpha=\dfrac{51}{50}, \ \beta=\dfrac{5}4,  \  \text{\rm when} \  n=4, \\ \\
&a=\dfrac{10}{11}, \  b=\dfrac{20}{21},  \  \alpha=\dfrac{31}{40}, \ \beta=\dfrac{207}{250},  \  \text{\rm when} \  n=5. \\
\end{cases}
\end{equation}
we have
\begin{equation}
\begin{aligned}
&b\bigl(-\dfrac{n(n-2)}2+\dfrac{n^2-4}4|dr|^2\bigl)+r^2f_{\min}(a, \alpha,\beta)
+2(n-1)\beta+2(n-2)\alpha\\
&-(n\beta+(n-1)\alpha)|dr|^2-((n-2)\beta-\alpha)r_1^2-(n-3)\alpha r_2^2\\
&\geq \varepsilon(n)=\min\{F(n, b, \alpha, \beta, 0), F(n, b, \alpha, \beta, 1)\}>0.\\
\end{aligned}
\end{equation}
Hence, we have
\begin{equation*}
\varepsilon(3)=\frac9{11}, \ \varepsilon(4)=\frac{377}{5260},  \  \ \varepsilon(5)=\frac{979826999}{65363627000}\approx 0.014999.
\end{equation*}
\end{proof}

\begin{theorem}
 For an $n$-dimensional complete
 two-sided $\delta$-stable minimal hypersurface $X:M^n\to \mathbf R^{n+1}$  in $\mathbf R^{n+1}$, with $3\leq n\leq 5$
 and $\delta>\delta_0(n)$,
 there exist an $\varepsilon(n)>0$ and smooth function $V$ such that
 $V\geq\varepsilon(n) -\tilde\Lambda_{(\alpha,\beta)}$ and
 $$
 b\int_N|\tilde \nabla\varphi |^2_{\tilde g}dv_{\tilde g}\geq \int_NV\varphi^2dv_{\tilde g}, \  \text {\rm for } \ \varphi \in \mathcal C^{1}_{c}(N),
 $$
 where $\tilde\Lambda_{(\alpha,\beta)}$ is the minimum of the  ${(\alpha,\beta)}$-bi-Ricci curvature
 of  $N=M\setminus X^{-1}(O)$ at each point.
\end{theorem}
\begin{proof}
Since $\tilde \omega_i=r^{-1}\omega_i$, we know
\begin{equation}\label{eq:3.8}
dv_{\tilde g}=r^{-n}dv_g, \ |\tilde \nabla f|_{\tilde g}^2 =r^2|\nabla f|_g^2.
\end{equation}
For $f\in \mathcal C^{1}_{c}(N)$,
\begin{equation}\label{eq:3.9}
\begin{aligned}
&\int_Nr^{n-2}|\tilde \nabla f |^2_{\tilde g}dv_{\tilde g}=\int_M|\nabla f |^2_{g}dv_{ g}\\
&\geq\delta  \int_MSf^2dv_{ g}=\delta  \int_Nr^nSf^2dv_{\tilde g}.
\end{aligned}
\end{equation}
Putting $f=r^{\frac{2-n}{2}}\varphi$, since
$$
\begin{aligned}
&|\tilde \nabla f|_{\tilde g}^2=r^{2-n}|\tilde \nabla\varphi|_{\tilde g}^2-(n-2)r^{1-n}\varphi\tilde g(\tilde \nabla r, \tilde\nabla \varphi)
+\dfrac{(n-2)^2}4r^{-n}|\tilde \nabla r|_{\tilde g}^2\varphi^2,
\end{aligned}
$$
and
$$
\tilde\Delta \log r=n-n|dr|^2,
$$
we have, from (\ref{eq:3.8}) and (\ref{eq:3.9}),
\begin{equation}
\begin{aligned}
&\int_N|\tilde \nabla \varphi |^2_{\tilde g}dv_{\tilde g}\\
=&\int_Nr^{n-2}|\tilde \nabla f |^2_{\tilde g}dv_{\tilde g}
-\int_N\bigl(\dfrac{n(n-2)}2-\dfrac{n^2-4}{4}|dr|^2\bigl)|\varphi^2dv_{\tilde g}\\
\ge &\delta  \int_Nr^2S\varphi^2dv_{\tilde g}-\int_N\bigl(\dfrac{n(n-2)}2-\dfrac{n^2-4}{4}|dr|^2\bigl)\varphi^2dv_{\tilde g}.\\
\end{aligned}
\end{equation}
By a direct calculation, we have
\begin{equation}
\begin{aligned}
&r^2{\rm Bi_{(\alpha,\beta)}Ric}_{12}+{\langle X,\vec e_{n+1}\rangle} \big[((n-2)\beta-\alpha)h_{11}+(n-3)\alpha h_{22}\big]\\
&=\widetilde{{\rm Bi}_{(\alpha,\beta)}{\rm Ric}}_{12}-2(n-1)\beta-2(n-2)\alpha\\
&+(n\beta+(n-1)\alpha)|dr|^2+((n-2)\beta-\alpha)r_1^2+(n-3)\alpha r_2^2.\\
\end{aligned}
\end{equation}
Hence, according to the lemma \ref{lem3.2},  we obtain, with $b\delta=a$,

\begin{equation}
\begin{aligned}
&b\bigl(\delta r^2S-\dfrac{n(n-2)}2+\dfrac{n^2-4}4|dr|^2\bigl)\\
&\geq b\bigl(-\dfrac{n(n-2)}2+\dfrac{n^2-4}4|dr|^2\bigl)+r^2f_{\min}(a, \alpha, \beta)\\
&-r^2{\rm Bi_{(\alpha,\beta)}Ric}_{12}-{\langle X,\vec e_{n+1}\rangle} \big[((n-2)\beta-\alpha)h_{11}+(n-3)\alpha h_{22}\big]\\
&=b\bigl(-\dfrac{n(n-2)}2+\dfrac{n^2-4}4|dr|^2\bigl)+r^2f_{\min}(a, \alpha, \beta)\\
&-\widetilde{{\rm Bi}_{(\alpha,\beta)}{\rm Ric}}_{12}+2(n-1)\beta+2(n-2)\alpha\\
&-(n\beta+(n-1)\alpha)|dr|^2-((n-2)\beta-\alpha)r_1^2-(n-3)\alpha r_2^2.\\
\end{aligned}
\end{equation}
We can assume that the local orthonormal frame is chosen
such that $\tilde \Lambda_{(\alpha,\beta)}=\widetilde{{\rm Bi}_{(\alpha,\beta)}{\rm Ric}}_{12}$.
We conclude, from the lemma \ref{lem3.3},
\begin{equation}
\begin{aligned}
&b\bigl(\delta r^2S-\dfrac{n(n-2)}2+\dfrac{n^2-4}4|dr|^2\bigl)\\
&\geq b\bigl(-\dfrac{n(n-2)}2+\dfrac{n^2-4}4|dr|^2\bigl)+r^2f_{\min}(a, \alpha, \beta)\\
&-\tilde\Lambda_{(\alpha,\beta)}+2(n-1)\beta+2(n-2)\alpha\\
&-(n\beta+(n-1)\alpha)|dr|^2-((n-2)\beta-\alpha)r_1^2-(n-3)\alpha r_2^2\\
&\geq \varepsilon(n) -\tilde\Lambda_{(\alpha,\beta)},\\
\end{aligned}
\end{equation}
where $\varepsilon(n)$ is defined by the formula (\ref{eq:3.2}).
Therefore, we obtain that  there is  a smooth function  $V\geq \varepsilon(n) -\tilde\Lambda_{(\alpha,\beta)}$ such that
 $$
 b\int_N|\tilde \nabla\varphi |^2_{\tilde g}dv_{\tilde g}\geq \int_NV\varphi^2dv_{\tilde g}, \  \text {\rm for } \ \varphi \in \mathcal C^{1}_{c}(N).
 $$
 This finishes the proof of the theorem 3.1.
\end{proof}

\vskip5mm
\section{On warped $\mu$-bubbles and Euclidean volume growth}
\vskip4mm
\noindent
Assume that  Riemannian manifold ($N^n, \tilde g$) satisfies
for $V\geq \varepsilon(n) -\tilde\Lambda_{(\alpha,\beta)}$ and
 \begin{equation}\label{eq4.1}
 b\int_N|\tilde \nabla\varphi |^2_{\tilde g}dv_{\tilde g}\geq \int_NV\varphi^2dv_{\tilde g}, \  \text {\rm for } \ \varphi \in \mathcal C^{1}_{c}(N).
 \end{equation}
According to  the results of Fischer-Colbrie and Schoen  in \cite{fs}, there exists a positive function $w$ on $N^n$ such that
$$
-b\tilde \Delta w=Vw.
$$
For  an $n$-dimensional complete two-sided $\delta$-stable minimal hypersurface in $\mathbf R^{n+1}$, its universal covering is also
 an $n$-dimensional complete two-sided $\delta$-stable minimal hypersurface.
 Hence, we can assume that $X: M^n\to \mathbf R^{n+1}$
 is simply connected.  Furthermore, by making use of the standard point-picking argument in \cite{w} (also see \cite{cl1} and \cite{hlw}),
 we can assume that an $n$-dimensional complete two-sided $\delta$-stable minimal
 hypersurface in $\mathbf R^{n+1}$, which we considered, has one end.  In fact, for $n=3$, in \cite{cmmr}, Catino, Mari, Mastrolia and
 Roncoroni have proved that $3$-dimensional complete two-sided
 $\delta$-stable minimal  hypersurfaces in $\mathbf R^{4}$ have one end if $\delta>\frac13$ .
 For reader's convenience, we give the  idea  for  proving.  For an  $n$-dimensional complete two-sided
 $\delta$-stable minimal hypersurface in $\mathbf R^{n+1}$, if there exists a constant $C$ such that, for any $R$,
 $$
\sqrt {S(p)}d(p, \partial (B_R(O))\leq C, \quad p \in   B_R(O),
$$
holds,
then $X: M^n\to \mathbf R^{n+1}$ is a hyperplane, where $B_R(O)$ denotes the geodesic ball with radius  $R$.
Otherwise,  there exist sequences $\{R_i\}$, $\{p_i\}$ and $\{a_i\}$ such that
$$
\sqrt {S(p_i)}d(p_i, \partial (B_{R_i}(O))=a_i\to \infty.
$$
We  assume that $p_i$ maximizes $\sqrt {S(p)}d(p, \partial (B_{R_i}(O))$,  $p\in B_{R_i}(O)$.
By rescaling for $B_{R_i}(O)$, we can assume $\sqrt {S_{M_i}(p_i)}=|A_{M_{i}}(p_i)|=1$, where
$M_i=B_{R_i}(O)$ with intrinsic metric.
For any $k<a_i$ and $p\in M_i$ with $d_{M_i}(p, p_i)<k$, we have
$$
\sqrt {S_{M_i}(p)}\leq \dfrac{a_i}{d(p, \partial (B_{R_i}(O))}\leq \dfrac{a_i}{a_i-k}\to 1.
$$
Thus, for any $k$,
 $$
\sup_{d_{M_i}(p, p_i)<k}\sqrt {S_{M_i}(p)}\leq \dfrac{a_i}{d(p, \partial (B_{R_i}(O))}\leq \dfrac{a_i}{a_i-k}\to 1.
$$
Thus, $M_i=B_{R_i}(O)$ with intrinsic metric smoothly converge  to a  complete
two-sided $\delta$-stable minimal hypersurface  $M_{\infty}$ if necessary, taking a subsequence.
Hence,  $M_{\infty}$ is a  complete
two-sided $\delta$-stable minimal hypersurface  with the bounded second fundamental form.
According to the results of Cheng and Zhou \cite{cz}, we know that $M_{\infty}$ has  one end.
Thus, we only need to prove that a  complete
two-sided $\delta$-stable minimal hypersurface  with one end is a hyperplane.\\
Letting $y_0>0$ be  a constant,  which only depends on dimension $n$ and will be  defined
in (\ref{eq:4.11}), we have
\begin{theorem}\label{thm4.1}
For $q$, $\alpha$, $\beta$ and $\varepsilon(n)$ defined in (\ref{eq:3.2}) and (\ref{eq:3.5}),
letting $\Omega_{+} $ be a domain in $N^n$ such that
$N \setminus \mathcal N_{\frac3{y_0}\pi}(\Omega_{+})\neq \emptyset$, there exists a domain $\Omega_{*} $ such that
$V>\frac{\varepsilon(n)}{2\alpha}- \bar \lambda$ and
$$
\dfrac{4}{4-q} \dfrac{\beta}{\alpha}\int_{\Sigma}|\bar \nabla f |^2_{\bar g}dv_{\bar g}
\geq \int_{\Sigma}Vf^2dv_{\bar g}, \  \text {\rm for } \ f \in \mathcal C^{1}_{c}(\Sigma),
 $$
 where $\Sigma =\partial \Omega_{*}$,  $\bar g$ is the induced metric on $\Sigma$ from $N$ and $\bar \lambda$ is the minimum
 of Ricci curvature at each point  in   $\Sigma$.
\end{theorem}
\begin{proof}
Let $\Omega_{-} $ and $\Omega_{+} $ be domains in $N^n$ such that
$$
\Omega_{+}\subset\subset  \Omega_{-} \subset\subset  \bar {\mathcal N}_{\frac{3}{ {y_0}}\pi}(\Omega_{+}).
$$
Taking a smooth function $h:\Omega_{-}\setminus \Omega_{+}\to \mathbf R$ such that
$$
\lim_{p\to\partial \Omega_{+}}h(p)=-\infty, \  \   \lim_{p\to\partial \Omega_{-}}h(p)=\infty.
$$
We take a domain $\Omega$ such that
$$
\Omega_{+}\subset\subset  \Omega \subset\subset\Omega_{-}
$$
with finite perimeter.
From (\ref{eq4.1}), we know that  there exists a smooth positive function $w$ such that
\begin{equation}\label{eq4.2}
-b\tilde \Delta w=Vw.
\end{equation}
We consider
\begin{equation*}
\mathcal A(\Omega)=\int_{\partial^*\Omega}w^q dA_{\bar g}-\int_{\Omega}w^qhdv_{\tilde g},
\end{equation*}
where $\partial^*\Omega$ is the reduced boundary of $\Omega$.
By the same arguments as in Chodosh and Li \cite{cl1} and Zhu \cite{z}, there exists a domain  $\Omega_{*}$ with finite perimeter,
which minimizes the functional $\mathcal A$ and $\partial^*\Omega_{*}=\Sigma\neq \emptyset$ is smooth, such that
$$
\Omega_{+}\subset\subset  \Omega_{*} \subset\subset\Omega_{-} \subset\subset  \bar {\mathcal N}_{\frac{3}{ {y_0}}\pi}(\Omega_{+}).
$$
We consider a variation $\Omega_{t}$ of  $\Omega_{*}$ generated by $\varphi \vec e_{n}$, where $\vec e_{n}$ denotes the
outward unit normal vector  field of $\Sigma$.
\begin{lemma}
The first variation formula is given by
\begin{equation}
0=\dfrac{d\mathcal A(\Omega_t)}{dt}\bigl|_{t=0}=\int_{\Sigma}(\bar H+q w^{-1}dw(\vec e_{n})
-h)w^q\varphi dA_{\bar g}, \ {\rm for \ any  } \ \   \varphi \in \mathcal C^{1}_{c}(\Sigma)
\end{equation}
and
\begin{equation}\label{eq4.4}
\bar H+q\  w^{-1}dw(\vec e_{n}) -h=0,
\end{equation}
where $\bar H$ denotes  the mean curvature of $\Sigma$ with respect to the induced metric $\bar g$.
\end{lemma}
\noindent
By  making use of
$$
\begin{aligned}
&\dfrac{d\bar {H}(t)}{dt}=-\bar \Delta \varphi-(|\bar B|^2+\widetilde{\rm Ric}(\vec e_{n}, \vec e_{n}))\varphi,\\
&\tilde\nabla_{n}\tilde\nabla_{n} w=\tilde {\Delta }w-\bar{\Delta }w -\bar H\tilde w_{n},\\
\end{aligned}
$$
we obtain the following second variation formula
\begin{lemma}\label{lem4.2}
\begin{equation}
\begin{aligned}
&0\leq \dfrac{d^2\mathcal A(\Omega_t)}{dt^2}\bigl|_{t=0}\\
&=\int_{\Sigma}w^q\bigl\{-\varphi \bar{\Delta}\varphi-(|\bar B|^2+\widetilde{\rm Ric}(\vec e_{n}, \vec e_{n}))\varphi^2\\
&+\big(qw^{-1}\big[\tilde\Delta w-\bar\Delta w-\bar H\tilde w_{n}\big]-qw^{-2}\big(dw(\vec e_{n})\big)^2\big)\varphi^2\\
& -qw^{-1}\bar g(\bar{ \nabla} w, \bar{\nabla} \varphi)\varphi
-dh(\vec e_{n})\varphi ^2\bigl\}dA_{\bar g}, \ {\rm for \ any  } \ \   \varphi \in \mathcal C^{1}_{c}(\Sigma),
\end{aligned}
\end{equation}
where $\bar B$ is the second fundamental form of $\Sigma$.
\end{lemma}
\noindent
For $f=w^{\frac q2}\varphi$, we have
$$
\bar \nabla f-\frac{q}2w^{-1}f \bar \nabla w =w^{\frac q2}\bar \nabla\varphi
$$
and, from Stokes theorem,
\begin{equation}\label{eq4.6}
\begin{aligned}
&\int_{\Sigma}w^q\bigl\{-\varphi \bar{\Delta}\varphi
-qw^{-1}\bar\Delta w\varphi^2 -qw^{-1}\bar g(\bar{ \nabla} w, \bar{\nabla} \varphi)\varphi \bigl\}dA_{\bar g}\\
&=\int_{\Sigma}\big\{|\bar\nabla f|^2+\dfrac{(q-4)q}4w^{-2}|\bar\nabla w|^2f^2+qw^{-1}f\bar g(\bar{ \nabla} w, \bar{\nabla}f)\bigl\}dA_{\bar g}\\
&\leq \dfrac{4}{4-q}\int_{\Sigma}|\bar\nabla f|^2dA_{\bar g}.\\
\end{aligned}
\end{equation}
From the lemma \ref{lem4.2}, (\ref{eq4.2}), (\ref{eq4.4}) and (\ref{eq4.6}), we obtain
$$
\begin{aligned}
&\dfrac{4}{4-q}\int_{\Sigma}|\bar\nabla f|^2dA_{\bar g}\\
&\geq \int_{\Sigma}\bigl(|\bar B|^2+\widetilde{\rm Ric}(\vec e_n,\vec e_n)+q(\tilde{\nabla}_{n}\log w)^2
-qw^{-1}\tilde{\Delta} w+q\bar H\tilde{\nabla}_{n}\log w+\tilde \nabla_{n}h\bigl)f^2dA_{\bar g}\\
&=\int_{\Sigma}\bigl\{|\bar B|^2+\widetilde{\rm Ric}(\vec e_n,\vec e_n)+\dfrac{q}bV
+q\big((\tilde{\nabla}_{n}\log w)^2+\bar H\tilde{\nabla}_{n}\log w\big)+\tilde \nabla_{n}h\bigl\}f^2dA_{\bar g}.\\
\end{aligned}
$$
From Gauss equation, we know
$$
\bar {\rm Ric}_{11}= \sum_{j=1}^{n-1}\bar R_{1j1j}=\sum_{j=1}^{n-1}[\tilde R_{1j1j}+(\bar B_{11}\bar B_{jj}-\bar B_{1j}^2)]
$$
\begin{equation}
\begin{aligned}
&\widetilde{\rm Ric}(\vec e_{n},\vec e_n)-\dfrac{1}{\beta}\tilde \Lambda_{(\alpha, \beta)}\\
&\geq \widetilde{\rm Ric}(\vec e_{n},\vec e_{n})-\dfrac{1}{\beta}\widetilde{{\rm Bi}_{(\alpha,\beta)}{\rm Ric}}(\vec e_{n}, \vec e_1)\\
&=-\dfrac{\alpha}{\beta}\sum_{j=2}^{n-1}\tilde R_{1j1j}
=\dfrac{\alpha}{\beta}\bigl[\sum_{j=1}^{n-1}(\bar B_{11}\bar B_{jj}-\bar B_{1j}^2)-\bar {\rm Ric}_{11}\bigl].
\end{aligned}
\end{equation}
Taking $q=\dfrac b{\beta}$ and $\bar \lambda=\bar{\rm Ric}_{11}$ and from (\ref{eq4.4})
\begin{equation*}
-q\  w^{-1}dw(\eta)=\bar H -h,
\end{equation*}
we have
\begin{equation}
\begin{aligned}
\dfrac{4 }{4-q}\int_{\Sigma}|\bar\nabla f|^2dA_{\bar g}
&\geq
\int_{\Sigma}\bigl\{|\bar B|^2+\widetilde{\rm Ric}(\vec e_n,\vec e_n)+\dfrac{q}bV\\
&+q\big((\tilde{\nabla}_{n}\log w)^2+\bar H\tilde{\nabla}_{n}\log w\big)+\tilde \nabla_{n}h\bigl\}f^2dA_{\bar g}\\
&\geq
\int_{\Sigma}\bigl\{|\bar B|^2+\dfrac{\varepsilon(n)}{\beta}-\dfrac{\alpha}{\beta}\bar \lambda +\dfrac{\alpha}{\beta}(\bar B_{11}\bar H-\sum_{j=1}^{n-1}\bar B_{1j}^2)\\
&+\dfrac {\beta}b(\bar H-h)^2-\bar H(\bar H-h)+\tilde \nabla_{n}h\bigl\}f^2dA_{\bar g}.\\
\end{aligned}
\end{equation}
At each point, we may take the orthonormal frame $\{\vec e_1, \cdots, \vec e_{n-1}\}$ such that
$\bar B_{ij}=\bar \lambda_i\delta_{ij} $ and $\bar \mu_i=\bar \lambda_i-\dfrac{1}{n-1}\bar H$. Thus, we have, for $\dfrac{\alpha}{\beta}<\dfrac{n-1}{n-2}$,
$$
\begin{aligned}
&|\bar B|^2+\dfrac{\alpha}{\beta}(\bar B_{11}\bar H-\sum_{j=1}^{n-1}\bar B_{1j}^2)\\
&=\sum_{i=1}^{n-1}\bar \mu_i^2+\dfrac{1}{n-1}\bar H^2
+\dfrac{\alpha}{\beta}\big(\mu_1\bar H+\dfrac{1}{n-1}\bar H^2-\bar\mu_1^2-\dfrac{2}{n-1}\bar H\bar\mu_1-\dfrac{1}{(n-1)^2}\bar H^2\big)\\
&\geq (\dfrac{n-1}{n-2}-\dfrac{\alpha}{\beta})\bar\mu_1^2+\frac{(n-3)\alpha}{(n-1)\beta}\bar H\bar \mu_1
+\dfrac1{n-1}(1+\dfrac{\alpha}{\beta}\dfrac{n-2}{n-1})\bar H^2\\
&\geq - \dfrac{(n-2)(n-3)^2\alpha^2}{4(n-1)^2\beta[(n-1)\beta-(n-2)\alpha]}\bar H^2
+\dfrac{(n-1)\beta+(n-2)\alpha}{(n-1)^2\beta}\bar H^2\\
&= \dfrac{4\beta^2-(n-2)\alpha^2}{4\beta[(n-1)\beta-(n-2)\alpha]}\bar H^2.\\
\end{aligned}
$$
We obtain
\begin{equation*}
\begin{aligned}
&\dfrac{4}{4-q}\int_{\Sigma}|\bar\nabla f|^2dA_{\bar g}\\
&\geq
\int_{\Sigma}\bigl\{\dfrac{\varepsilon(n)}{\beta}-\dfrac{\alpha}{\beta}\bar \lambda + \dfrac{4\beta^2-(n-2)\alpha^2}{4\beta[(n-1)\beta-(n-2)\alpha]}\bar H^2\\
&+\dfrac {\beta}b(\bar H-h)^2-\bar H(\bar H-h)+\tilde \nabla_{n}h\big)\bigl\}f^2dA_{\bar g}\\
&=\int_{\Sigma}\bigl\{\dfrac{\varepsilon(n)}{\beta}-\dfrac{\alpha}{\beta}\bar \lambda
+\bigl( \dfrac{4\beta^2-(n-2)\alpha^2}{4\beta[(n-1)\beta-(n-2)\alpha]}+\dfrac {\beta}b-1\bigl)\bar H^2\\
&+(1-\dfrac {2\beta}b) \bar H h+\dfrac {\beta}bh^2 +\tilde \nabla_{n}h\big)\bigl\}f^2dA_{\bar g}\\
&\geq \int_{\Sigma}\bigl\{\dfrac{\varepsilon(n)}{\beta}-\dfrac{\alpha}{\beta}\bar \lambda
+\bigl( \dfrac{4\beta^2-(n-2)\alpha^2}{4\beta[(n-1)\beta-(n-2)\alpha]}+\dfrac {1}q-1-L(n)|\frac{1}2-\frac {1}q|\bigl)\bar H^2\\
&+\big[\dfrac {1}q-\dfrac1{L(n)}|\frac{1}2-\frac {1}q|\big]h^2 +\tilde \nabla_{n}h\big)\bigl\}f^2dA_{\bar g},\\
\end{aligned}
\end{equation*}
that is.
\begin{equation}
\begin{aligned}
&\dfrac{4 }{4-q}\dfrac{\beta}{\alpha}\int_{\Sigma}|\bar\nabla f|^2dA_{\bar g}\\
&\geq
 \int_{\Sigma}\bigl\{\dfrac{\varepsilon(n)}{\alpha}-\bar \lambda
+\big[\dfrac {1}q-\dfrac1{L(n)}|\frac{1}2-\frac {1}q|\big]\dfrac{\beta}{\alpha}h^2
+ \dfrac{\beta}{\alpha}\tilde \nabla_{n}h\big)\bigl\}f^2dA_{\bar g},\\
\end{aligned}
\end{equation}
where $L(n)$ is given by $L(3)=\frac{71}{11}$,  $L(4)=\frac{189697}{206625}\approx  0.9181$,
$L(5)=\frac{106986857}{251572482}\approx  0.4253$. \\
According to the following lemma, we have
\begin{equation}\label{eq4.10}
\begin{aligned}
&\dfrac{4 }{4-q}\dfrac{\beta}{\alpha}\int_{\Sigma}|\bar\nabla f|^2dA_{\bar g}
\geq
 \int_{\Sigma}\bigl(\dfrac{\varepsilon(n)}{2\alpha}-\bar \lambda \bigl)f^2dA_{\bar g}.\\
\end{aligned}
\end{equation}
\end{proof}
\noindent
Taking $\gamma_0(n)=\big\{\frac1q-\frac{1}{L(n)}(\frac1q-\frac12\big\}\frac{\beta}{\alpha}$, that is,  $\gamma_0(3)=\frac{77}{142}$,
$\gamma_0(4)=\frac{276875}{569091}\approx 0.48652 $, $\gamma_0(5)=\frac{667989}{855894856}\approx 0.00078$.

\begin{lemma}
There exists a smooth function $h$ such that
$$
\dfrac{\varepsilon(n)}{2\alpha}+ \gamma_0(n)h^2+ \dfrac{\beta}{\alpha}\tilde \nabla_{n}h\geq 0.
$$
\end{lemma}
\begin{proof}
Defining
\begin{equation}\label{eq:4.11}
x_0=x_0(n)=\sqrt{\dfrac{\varepsilon(n)}{2\alpha\gamma_0(n)}}, \  \
y_0=y_0(n)=\dfrac1{2\sqrt2\beta}\sqrt{\alpha\varepsilon(n)\gamma_0(n)}
\end{equation}
and $-\eta(t)=x_0\tan(y_0t-\frac{\pi}2)$,   $t \in (0, \frac1{y_0}\pi)$,  we obtain
$$
-\eta^{\prime}(t)=x_0y_0+\dfrac{y_0}{x_0}\eta(t)^2,  \ \ t \in (0, \frac1{y_0}\pi).
$$
Let $F: N^n\setminus \Omega_{+}\to \mathbf R$ be a smoothing of the distance function $d_{\tilde g}(\cdot, \partial \Omega_{+})$
such that
$$
\dfrac12d_{\tilde g}(\cdot, \partial \Omega_{+})\leq F(\cdot)\leq 2d_{\tilde g}(\cdot, \partial \Omega_{+}) \ {\rm and }\ \ |\tilde \nabla F|^2\leq 4.
$$
Taking $\xi_0>0$ very  small such that $(1+\xi_0)\frac1{y_0}\pi$ is a regular value of $F$. Define $\Omega_{-}$ by
$$
\Omega_{-}=\Omega_{+}\cup\{p\in N^n\setminus \Omega_{+};\ \  F(p)\leq (1+\xi_0)\frac1{y_0}\pi\}
$$
On $\Omega_{-}$,
$$
d_{\tilde g}(\cdot, \partial \Omega_{+})\leq 2(1+\xi_0)\frac1{y_0}\pi\leq \dfrac{3}{ {y_0}}\pi.
$$
Hence,  we have
$$
\Omega_{-}\subset \mathcal N_{\frac{3}{ {y_0}}\pi}(\Omega_{+}).
$$
On $\{p\in N^n\setminus \Omega_{+}; \ \ 0<F(p)< (1+\xi_0)\frac1{y_0}\pi\}$,  we  define $h(p)=\eta(\dfrac{F(p)}{1+\xi_0})$.
We have $\lim_{p\to\partial\Omega{\pm}}h(p)=\pm\infty$.
Since
$$
-\eta^{\prime}(s)= x_0y_0+ \dfrac{y_0}{x_0}\eta(s)^2,
$$
we have
$$
\begin{aligned}
&\dfrac{\beta}{\alpha}|\tilde \nabla_{n}h|\\
&=\dfrac{\beta}{\alpha}|\eta^{\prime}(\dfrac{F(p)}{1+\xi_0})\dfrac{1}{1+\xi_0}\tilde \nabla_{n}F(p)|\\
&\leq 2 \dfrac{\beta}{\alpha}\big(x_0y_0+ \dfrac{y_0}{x_0}\eta(\dfrac{F(p)}{1+\xi_0})^2\big)\dfrac{1}{1+\xi_0}\\
&\leq  2 \dfrac{\beta}{\alpha}\big(x_0y_0+ \dfrac{y_0}{x_0}h^2\big).
\end{aligned}
$$
Because of  $x_0=\sqrt{\dfrac{\varepsilon(n)}{2\alpha\gamma_0(n)}}$,  $ y_0=\dfrac1{2\sqrt2\beta}\sqrt{\alpha\varepsilon(n)\gamma_0(n)}$,
we have
$$
2 \dfrac{\beta}{\alpha}x_0y_0=\dfrac{\varepsilon(n)}{2\alpha},  \ \ 2 \dfrac{\beta}{\alpha} \dfrac{y_0}{x_0}= \gamma_0(n).
$$
Therefore, we obtain
$$
\dfrac{\beta}{\alpha}|\tilde \nabla_{n}h|\leq \dfrac{\varepsilon(n)}{2\alpha}+ \gamma_0(n)h^2.
$$
\end{proof}
\noindent

\begin{theorem}
For an $n$-dimensional complete
 two-sided $\delta$-stable minimal hypersurface $X:M^n\to \mathbf R^{n+1}$  in $\mathbf R^{n+1}$
 with $3\leq n\leq 5$ and $\delta>\delta_0(n)$, the geodesic ball $B_R(p_0)$ of radius $R>0$ centered at some point $p_0$
 in $M^n$ satisfies
 $$
 \text{\rm vol }\{B_{R} (p_0)\} \leq \Lambda R^n,
 $$
 where $\Lambda$ is constant and
 \begin{equation*}
\begin{aligned}
&\delta_0 (3)=\frac13, \ \delta_0 (4)=\frac12,  \  \ \delta_0(5)=\frac{21}{22}.\\
\end{aligned}
\end{equation*}
\end{theorem}

\begin{proof}
Let $\Omega_{+}$ be a smooth domain in $M^n$ such that $B_R(p_0)\subset \Omega_{+}\subset B_{2R}(p_0)$ and
$O\notin X(\partial \Omega_{+})$. According to the Gulliver-Lawson conformal transformation $\tilde g=r^{-2}g$ and
theorem \ref{thm4.1},  we know that there exists a domain $\Omega_{*}$ in $M^n$ such that
$$
\Omega_{+}\subset\Omega_{*}\subset \mathcal N_{\frac3{y_0}\pi}(\Omega_{+})
$$
and  $\Sigma=\partial \Omega_{*}$ satisfies (\ref{eq4.10}), which is called the spectral Ricci curvature lower bound for the metric $\tilde g$.
We consider the connected component $\Omega_{**}$ of $\Omega_{*}$. Since $X:M^n\to \mathbf R^{n+1}$
is simply connected and has only one end, the unbounded component
of $M^n\setminus \Omega_{**}$ has only boundary component $\Sigma_0$.
Let $\Omega_1$ denote the bounded component of $M^n\setminus\Sigma_0$. We have
$$
B_R(p_0)\subset \Omega_{1}, \ \ \partial \Omega_1\subset \mathcal N_{\frac3{y_0}\pi}(B_{2R}(p_0)).
$$
Since  the Euclidean distance function $r$ is bounded by $2R$ on $\partial B_{2R}(p_0)$, making use of the lemma 6.2 in \cite{cl3},
on $\mathcal N_{\frac3{y_0}\pi}(B_{2R}(p_0)$, we have
$$
r\leq 2R\exp (\frac3{y_0}\pi).
$$
Because  the spectral Ricci curvature bound (\ref{eq4.10}) holds and, for $n=4, \ 5$,   we have
$$
0< \frac{4 }{4-q}\frac{\beta}{\alpha}\leq \frac{n-2}{n-3}.
$$
 From  estimate of area due to Antonelli and Xu \cite{ax1}, we obtain
$$
{\rm Area}_{\bar g}(\Sigma_0)\leq (\dfrac{(n-2)\alpha}{\varepsilon(n)})^{\frac{n-1}2}{\rm Area }(S^{n-1}).
$$
Thus, the  area of $\Sigma_0$ with respect to the induced metric $g$ satisfies
$$
{\rm Area}(\Sigma_0)\leq (\dfrac{(n-2)\alpha}{\varepsilon(n)})^{\frac{n-1}2}{\rm Area }(S^{n-1}) (2R\exp (\frac3{y_0}\pi))^{n-1}.
$$
The isoperimetric inequality for minimal hypersurfaces in $\mathbf R^{n+1}$ yields
$$
{\rm Vol}_g(B_R(p_0)\leq {\rm Vol}_g(\Omega_1)\leq
(\dfrac{(n-2)\alpha}{\varepsilon(n)})^{\frac{n}2}{\rm Vol }(B^{n}) (2R\exp (\frac3{y_0}\pi))^{n}.
$$
Thus, $X:M^n\to \mathbf R^{n+1}$ has the Euclidean volume growth.
For $n=3$,  by making use of the Gauss-Bonnet theorem, and following the same ideas
as in \cite{cl3, hlw} (see the remark in \cite{ax1} also), we also can prove that $X:M^n\to \mathbf R^{n+1}$
has the Euclidean volume growth.
\end{proof}

\section{Proof of theorems}
\vskip5mm
\noindent
In this section, we shall prove our theorems.
\vskip2mm
\noindent
{\it Proof of theorem 1.1.}
For $k_l=k(1-\frac{1}{2^{l-1}})\geq 0$ and $R_l=\frac R2(1+\frac{1}{2^{l-1}})$, we know
\begin{equation}
k_l<k_{l+1}, \ \ \lim_{l\to\infty}k_l=k, \  \ R_l>R_{l+1}, \ \ \lim_{l\to\infty}R_l=\dfrac{R}2.
\end{equation}
Since $X:M^n\to \mathbf R^{n+1}$ is $\delta$-stable, for any function $f\in \mathcal C_{c}^1(M)$, we have
\begin{equation}\label{eq:5.2}
\begin{aligned}
&\int_M|\nabla f |^2dv&\geq\delta  \int_MSf^2dv.
\end{aligned}
\end{equation}
Putting $M_{l, R}=M\cap \{u>k_l\}\cap B_R(p_0)$,  $M_{l}=M\cap \{u>k_l\}$, $u=S^{\frac{q}2}$, we have
\begin{equation}\label{eq:5.3}
\begin{aligned}
&\delta  \int_MS\big[(u-k_l)^+f\big]^2dv\leq \int_M|\nabla \big[(u-k_l)^+f\big] |^2dv\\
&= \int_{M_l}|\nabla u|^2f^2dv+\dfrac12 \int_{M_l}\nabla [(u-k_l)^+]^2\cdot \nabla f^2dv+ \int_{M_l} [(u-k_l)^+]^2|\nabla f|^2dv\\
&= \int_{M_l}|\nabla u|^2f^2dv-\dfrac12 \int_{M_l}\Delta [(u-k_l)^+]^2 f^2dv+ \int_{M_l} [(u-k_l)^+]^2|\nabla f|^2dv.\\
\end{aligned}
\end{equation}
In $\{u>k_l\}$,  since
$$
\nabla u=\dfrac{q}2S^{\frac{q}2-1}\nabla S, \ \ \dfrac12\Delta S=|\nabla A|^2-S^2,
$$
we have, by making use of  $|\nabla A|^2\geq (1+\frac2n)\frac1{4S}|\nabla S|^2$,
\begin{equation*}
\begin{aligned}
&\dfrac12\Delta [(u-k_l)^+]^2= (u-k_l)^+\Delta u+|\nabla u|^2,\\
&\Delta u=\Delta S^{\frac{q}2}=\frac{q}2S^{\frac{q}2-1}\Delta S+\frac{q}2(\frac{q}2-1)S^{\frac{q}2-2}|\nabla S|^2\\
&=q(\sqrt S)^{q-2}\big\{|\nabla A|^2-S^2\big\}+\dfrac{q-2}{qu}|\nabla u|^2\\
&\geq qS^{\frac q2-1}\big\{(1+\dfrac2n)\dfrac1{4S}|\nabla S|^2-S^2\big\}+\dfrac{q-2}{qu}|\nabla u|^2\\
&=(1+\dfrac2{n})\dfrac1{qu}|\nabla u|^2-qSu+\dfrac{q-2}{qu}|\nabla u|^2,\\
\end{aligned}
\end{equation*}
\begin{equation}\label{eq:5.4}
\begin{aligned}
&\dfrac12\Delta [(u-k_l)^+]^2= (u-k_l)^+\Delta u+|\nabla u|^2,\\
&\geq \dfrac1{q}(q-\dfrac{n-2}n)(1-\dfrac{k_l}u)|\nabla u|^2-q(u-k_l)^+Su+|\nabla u|^2.\\
\end{aligned}
\end{equation}
From (\ref{eq:5.3}) and (\ref{eq:5.4}), we obtain, for $\dfrac{n-2}n<q<\delta$,
\begin{equation}\label{eq:5.5}
\begin{aligned}
& \dfrac1{q}(q-\dfrac{n-2}n)\int_{M_l}(1-\dfrac{k_l}u)|\nabla u|^2f^2dv\\
&\leq  \int_{M_l} [(u-k_l)^+]^2|\nabla f|^2dv+ \int_MS\big[(q-\delta) \{(u-k_l)^+\}^2+qk_l(u-k_l)\big]f^2dv\\
&\leq  \int_{M_l} [(u-k_l)^+]^2|\nabla f|^2dv+ \dfrac{q^2k_l^2}{4(\delta -q)}\int_{M_l}u^{\frac{2}{q}}f^2dv.\\
\end{aligned}
\end{equation}
Since $k_{l+1}>k_l$, $\{u>k_{l+1}\}\subset\{u>k_{l}\}$ and in $\{u>k_{l+1}\}$
$$
1-\dfrac{k_l}u>1-\dfrac{k_l}{k_{l+1}}>\dfrac1{2^l}, \ \ u-k_l>k_{l+1}-k_l=\dfrac{k}{2^{l}}
$$
and since $x^s$ is a convex function for $s\geq 2$ and $x>0$, in $\{u>k_{l}\}$
$$
u^{\frac{2}q}=\big[(u-k_l)+k_l\big]^{\frac2q}\leq 2^{\frac2q-1}[(u-k_l)^{\frac2q}+k_l^{\frac2q}],
$$
we derive in view of  (\ref{eq:5.5}) and the above inequalities
\begin{equation}\label{eq:5.6}
\begin{aligned}
& \dfrac1{q}(q-\dfrac{n-2}n)\dfrac1{2^l}\int_{M_{l+1}}|\nabla u|^2f^2dv\\
&\leq  \int_{M_l} [(u-k_l)^+]^2|\nabla f|^2dv+ \dfrac{q^2k_l^2}{4(\delta -q)}2^{\frac2q-1}\int_{M_l}[(u-k_l)^{\frac2q}+k_l^{\frac2q}]f^2dv.\\
\end{aligned}
\end{equation}
Because of $\dfrac{n-2}n<q<\delta$, we know $\dfrac2q<\dfrac{2n}{n-2}$.
In $\{u>k_{l}\}$,
\begin{equation}\label{eq:5.7}
\begin{aligned}
 &u-k_{l-1}>k_{l}-k_{l-1}=\dfrac{k}{2^{l-1}},\\
 (u-k_l)^{\frac2q}&
 \leq\dfrac{(u-k_l)^{\frac2q}}{(u-k_{l-1})^{\frac2q}}\dfrac{(u-k_{l-1})^{\frac{2n}{n-2}}}{(\dfrac{k}{2^{l-1}})^{\frac{2n}{n-2}-\frac2q}}
 \leq{\big(\dfrac{2^{l-1}}k\big)^{\frac{2n}{n-2}-\frac2q}}{(u-k_{l-1})^{\frac{2n}{n-2}}},\\
 (u-k_l)^{2}&
 \leq\dfrac{(u-k_l)^{2}}{(u-k_{l-1})^{2}}\dfrac{(u-k_{l-1})^{\frac{2n}{n-2}}}{(\dfrac{k}{2^{l-1}})^{\frac{4}{n-2}}}
 \leq{\big(\dfrac{2^{l-1}}k\big)^{\frac{4}{n-2}}(u-k_{l-1})^{\frac{2n}{n-2}}},\\
 1&\leq {\big(\dfrac{2^{l-1}}k\big)^{\frac{2n}{n-2}}}{(u-k_{l-1})^{\frac{2n}{n-2}}}.
 \end{aligned}
 \end{equation}
 Hence, we obtain, for each $l$, in place of $f$ by $f_l$,
\begin{equation*}
\begin{aligned}
& \dfrac1{q}(q-\dfrac{n-2}n)\dfrac1{2^l}\int_{M_{l+1}}|\nabla u|^2f_l^2dv\leq  \int_{M_l} [(u-k_l)^+]^2|\nabla f_l|^2dv\\
&+ \dfrac{q^2k_l^2}{4(\delta -q)}2^{\frac2q-1} \big(\dfrac{2^{l-1}}k\big)^{\frac{2n}{n-2}-\frac2q}
\big[1+(1-\dfrac1{2^{l-1}})^{\frac2q} \big]
\int_{M_{l-1}}(u-k_{l-1})^{\frac{2n}{n-2}}f_l^2dv.\\
\end{aligned}
\end{equation*}
Since, in $\{u>k_{l+1}\}$,
$$
|\nabla( (u-k_l)^+f_l)|^2\leq 2|\nabla u|^2f_l^2+2[(u-k_l)^+]^2|\nabla f_l|^2,
$$
we have
\begin{equation*}
\begin{aligned}
&\int_{M_{l+1}}|\nabla( (u-k_l)^+f_l)|^2dv\\
&\leq   \big(\dfrac q{(q-\frac{n-2}n)}2^{l+1}+2\big)\int_{M_l} [(u-k_l)^+]^2|\nabla f_l|^2dv\\
&+  \dfrac {q2^{l+1}}{(q-\frac{n-2}n)}\dfrac{q^2k_l^2}{4(\delta -q)}2^{\frac2q-1} \big(\dfrac{2^{l-1}}k\big)^{\frac{2n}{n-2}-\frac2q}
\big[1+(1-\dfrac1{2^{l-1}})^{\frac2q} \big]
\int_{M_{l-1}}(u-k_{l-1})^{\frac{2n}{n-2}}f_l^2dv.\\
\end{aligned}
\end{equation*}
According to the Michael-Simon inequality, we conclude
\begin{equation}\label{eq:5.8}
\begin{aligned}
&\dfrac{1}{C_{MS}}\biggl\{\int_{M_{l+1}} \big[(u-k_{l+1})^+f_l\big]^{\frac{2n}{n-2}}dv\biggl\}^{\frac{n-2}n}\\
&\leq   \big(\dfrac q{(q-\frac{n-2}n)}2^{l+1}+2\big)\int_{M_l} [(u-k_l)^+]^2|\nabla f_l|^2dv\\
&+  \dfrac {q^32^{\frac2q}}{(q-\frac{n-2}n)(\delta -q)}k_l^22^{l-1} \big(\dfrac{2^{l-1}}k\big)^{\frac{2n}{n-2}-\frac2q}
\int_{M_{l-1}}(u-k_{l-1})^{\frac{2n}{n-2}}f_l^2dv.\\
\end{aligned}
\end{equation}
Taking, for each $l$, the function $f_l$ such that  $f_l=1$ in $M_{l+1,R_{l+1}}$, $f_l=0$ in $M\setminus M_{l,R_{l}}$,  $0\leq f_l \leq 1$ and
$|\nabla f_l|\leq \frac2{R_l-R_{l+1}}=\frac{2^{l+2}}{R}$,  we get from (\ref{eq:5.7}) and (\ref{eq:5.8})
\begin{equation*}
\begin{aligned}
&\dfrac{1}{C_{MS}}\biggl\{\int_{M_{l+1,R_{l+1}}} \big[(u-k_{l+1})^+\big]^{\frac{2n}{n-2}}dv\biggl\}^{\frac{n-2}n}\\
&\leq   \big(\dfrac q{(q-\frac{n-2}n)}2^{l+1}+2\big)\frac{2^{2l+4}}{R^2} {\big(\dfrac{2^{l-1}}k\big)^{\frac{4}{n-2}}\int_{M_{l,R_l}}(u-k_{l-1})^{\frac{2n}{n-2}}}dv\\
&+ \dfrac{q^32^{\frac2q}k^{\frac2q-\frac{4}{n-2}}}{(\delta -q)(q-\frac{n-2}n)} \big(2^{\frac{2n}{n-2}-\frac2q+1}\big)^{l-1}
\int_{M_{l-1,R_{l-1}}}(u-k_{l-1})^{\frac{2n}{n-2}}dv\\
\end{aligned}
\end{equation*}
Setting $k=\frac1{R^{\frac{n-2}n}}$ and
$$
S_l=\int_{M_{l,R_l}} \big[(u-k_l)^+\big]^{\frac{2n}{n-2}}dv,
$$
we have
\begin{equation*}
\begin{aligned}
&\dfrac{1}{C_{MS}}S_{l+1}^{\frac{n-2}n}\\
&\leq  \biggl\{ \big(\dfrac {2q}{(q-\frac{n-2}n)}+1\big)2^7\frac{\bigl(2^{\frac{3n+2}{n-2}}\bigl)^{l-1}}{R^{\frac{2n-4}{n}}}
+ \dfrac{q^32^{\frac2q}}{(\delta -q)(q-\frac{n-2}n)} \dfrac{\big(2^{\frac{2n}{n-2}-\frac2q+1}\big)^{l-1}}{R^{\frac{2(n-2)}{nq}-\frac{4}{n}}}\biggl\}
S_{l-1}\\
&\leq \biggl\{ \big(\dfrac {2q}{(q-\frac{n-2}n)}+1\big)2^7\frac{1}{R^{\frac{2n-4}{n}}}
+ \dfrac{q^32^{\frac2q}}{(\delta -q)(q-\frac{n-2}n)} \dfrac{1}{R^{\frac{2(n-2)}{nq}-\frac{4}{n}}}\biggl\}C^{l-1}S_{l-1}\\
&=\dfrac{C_0}{C_{MS}}C^{l-1}S_{l-1},
\end{aligned}
\end{equation*}
where
$$
\begin{aligned}
&C=\max\{2^{\frac{3n+2}{n-2}},  2^{\frac{2n}{n-2}-\frac2q+1}\}>1,\\
&C_0= C_{MS}\biggl\{\big(\dfrac {2q}{(q-\frac{n-2}n)}+1\big)2^7\frac{1}{R^{\frac{2n-4}{n}}}
+ \dfrac{q^32^{\frac2q}}{(\delta -q)(q-\frac{n-2}n)} \dfrac{1}{R^{\frac{2(n-2)}{nq}-\frac{4}{n}}}\biggl\}.
\end{aligned}
$$
We derive a recursion formula
\begin{equation}
\begin{aligned}
&S_{l+1}\leq C_0^{\frac{n}{n-2}}(C^{\frac n{n-2}})^{l-1}S_{l-1}^{\frac n{n-2}}.
\end{aligned}
\end{equation}
Thus, we have
\begin{equation*}
\begin{aligned}
&S_{2l+1}\leq C_0^{\frac{n}{n-2}}(C^{\frac n{n-2}})^{2l-1}S_{2l-1}^{\frac n{n-2}}.
\end{aligned}
\end{equation*}
We obtain
\begin{equation*}
\begin{aligned}
S_{2l+1}&\leq \bigl(C_0^{\frac{n}{n-2}}\big)^{\sum_{j=0}^{l-1}(\frac{n}{n-2})^j}
\big(C^{\frac{n}{n-2}}\bigl)^{\sum_{j=0}^{l-1}(2(l-j)-1)(\frac n{n-2})^j}\bigl(S_{1}\bigl)^{(\frac n{n-2})^l}\\
&\leq   \bigl(C_0^{\frac{n}2}C^{\frac{n^2}2}S_1\bigl)^{(\frac{n}{n-2})^l}.
\end{aligned}
\end{equation*}
Since $\frac{n-2}{n}<q$ and for sufficiently large $R$
$$
 \dfrac{\int_{B_{R} (p_0)}S^{\frac{qn}{n-2}}dv}{R^{\frac{(n-2)}{q}-2}}\leq \varepsilon_1,
 $$
 takeing  $\varepsilon_1$ such that
$$
\varepsilon_1C^{\frac{n^2}2}C_{MS}\biggl\{\big(\dfrac {2q}{(q-\frac{n-2}n)}+1\big)2^7
+ \dfrac{q^32^{\frac2q}}{(\delta -q)(q-\frac{n-2}n)} \biggl\}^{\frac{n}2}<1,
$$
$C_0^{\frac{n}2}C^{\frac{n^2}2}S_1<1$, we get $S_{2l+1}\to 0$. Because $\{S_l\}$ is a monotonous  sequence,
we obtain $S_{l}\to 0$. Hence, $X:M^n\to \mathbf R^{n+1}$   is a hyperplane. This completes  the proof of theorem 1.1.
\begin{flushright}
 $\square$
\end{flushright}

\vskip2mm
\noindent
{\it Proof of theorem 1.2}.
Since $X:M^n\to \mathbf R^{n+1}$ is complete two-sided $\delta$-stable minimal hypersurface,
for any function $f\in \mathcal C_{c}^1(M)$, we have
\begin{equation}\label{eq:5.10}
\begin{aligned}
&\delta  \int_MSf^2dv\leq \int_M|\nabla f |^2dv.
\end{aligned}
\end{equation}
Taking $(S+\epsilon)^k$ with $\epsilon>0$, we have
\begin{equation}\label{eq:5.11}
\begin{aligned}
&\delta  \int_MS(S+\epsilon)^{2k}f^2dv\leq \int_M|\nabla ((S+\epsilon)^k f) |^2dv\\
&=\int_M(S+\epsilon)^{2k}|\nabla f |^2dv+2k\int_M(S+\epsilon)^{2k-1}f\nabla S\cdot \nabla f dv\\
&+k^2\int_Mf^2(S+\epsilon)^{2(k-1)}|\nabla S |^2dv.\\
\end{aligned}
\end{equation}
According to the Simons formula
$$
\dfrac12\Delta S=|\nabla A|^2-S^2\geq (1+\dfrac{2}n)\dfrac{4}{S+\epsilon}|\nabla S|^2-S^2,
$$
\begin{equation*}
\begin{aligned}
&-\dfrac{(2k-1)}2\int_M(S+\epsilon)^{2k-2}f^2|\nabla S|^2dv-\int_M(S+\epsilon)^{2k-1}f\nabla S\cdot \nabla f dv\\
&\geq (1+\dfrac{2}n)\dfrac14\int_Mf^2(S+\epsilon)^{2k-2}|\nabla S |^2dv-\int_{M}(S+\epsilon)^{2k}Sf^2\\
\end{aligned}
\end{equation*}
that is,
\begin{equation}\label{eq:5.12}
\begin{aligned}
&(k+\dfrac1{2n}-\dfrac14)\int_M(S+\epsilon)^{2k-2}f^2|\nabla S|^2dv\\
&\leq \int_{M}(S+\epsilon)^{2k}Sf^2-\int_M(S+\epsilon)^{2k-1}f\nabla S\cdot \nabla f dv\\
\end{aligned}
\end{equation}
If $\int_M(S+\epsilon)^{2k-1}f\nabla S\cdot \nabla f dv\leq 0$, we have
\begin{equation}\label{eq:5.13}
\begin{aligned}
&\delta  \int_MS(S+\epsilon)^{2k}f^2dv\\
&\leq \int_M(S+\epsilon)^{2k}|\nabla f |^2dv
+k^2\int_Mf^2(S+\epsilon)^{2(k-1)}|\nabla S |^2dv.\\
\end{aligned}
\end{equation}
In view of,  for $s>0$,
\begin{equation}
\begin{aligned}
&-2\int_M(S+\epsilon)^{2k-1}f\nabla S\cdot \nabla f dv\\
&\leq s\int_M(S+\epsilon)^{2k}|\nabla f |^2dv
+\dfrac1s\int_Mf^2(S+\epsilon)^{2(k-1)}|\nabla S |^2,
\end{aligned}
\end{equation}
we have
\begin{equation}\label{eq:5.15}
\begin{aligned}
&(2k+\dfrac1{n}-\dfrac12-\dfrac1s)\int_M(S+\epsilon)^{2k-2}f^2|\nabla S|^2dv\\
&\leq 2\int_{M}(S+\epsilon)^{2k}Sf^2dv+s\int_M(S+\epsilon)^{2k}|\nabla f |^2dv.\\
\end{aligned}
\end{equation}
From (\ref{eq:5.13}) and (\ref{eq:5.15}), we obtain
\begin{equation}\label{eq:5.16}
\begin{aligned}
&\biggl\{(2k+\dfrac1{n}-\dfrac12-\dfrac1s)\dfrac{\delta}{k^2}-2\bigl\}\int_{M}(S+\epsilon)^{2k}Sf^2dv\\
&\leq \bigl\{s+\dfrac{(2k+\dfrac1{n}-\dfrac12-\dfrac1s)}{k^2}\bigl\}\int_M(S+\epsilon)^{2k}|\nabla f |^2dv.\\
\end{aligned}
\end{equation}
If $s\to \infty$,
$$
(2k+\dfrac1{n}-\dfrac12-\dfrac1s)\dfrac{\delta}{k^2}-2\to (2k+\dfrac1{n}-\dfrac12)\dfrac{\delta}{k^2}-2.
$$
We take  $k$ such that $\delta-\sqrt{\delta(\delta-\frac{n-2}n)}<2k<\delta+\sqrt{\delta(\delta-\frac{n-2}n)}$,
we know
$$
(2k+\dfrac1{n}-\dfrac12-\dfrac1s)\dfrac{\delta}{k^2}-2>0
$$
for a sufficient large  $s$.
Hence, we obtain
\begin{equation}\label{eq:5.17}
\begin{aligned}
&\int_{M}(S+\epsilon)^{2k}Sf^2dv
\leq C\int_M(S+\epsilon)^{2k}|\nabla f |^2dv\\
\end{aligned}
\end{equation}
If $\int_M(S+\epsilon)^{2k-1}f\nabla S\cdot \nabla f dv>0$, we have
\begin{equation}\label{eq:5.18}
\begin{aligned}
&\delta  \int_MS(S+\epsilon)^{2k}f^2dv\\
&\leq (1+s_1) \int_M(S+\epsilon)^{2k}|\nabla f |^2dv
+k^2(1+\dfrac{1}{s_1})\int_Mf^2(S+\epsilon)^{2(k-1)}|\nabla S |^2dv\\
\end{aligned}
\end{equation}
in view of,  for $s_1>0$,
\begin{equation*}
\begin{aligned}
&2k\int_M(S+\epsilon)^{2k-1}f\nabla S\cdot \nabla f dv\\
&\leq s_1\int_M(S+\epsilon)^{2k}|\nabla f |^2dv
+\dfrac{k^2}{s_1}\int_Mf^2(S+\epsilon)^{2(k-1)}|\nabla S |^2
\end{aligned}
\end{equation*}
and
\begin{equation}\label{eq:5.19}
\begin{aligned}
&(k+\dfrac1{2n}-\dfrac14)\int_M(S+\epsilon)^{2k-2}f^2|\nabla S|^2dv\\
&\leq \int_{M}(S+\epsilon)^{2k}Sf^2dv\\
\end{aligned}
\end{equation}
according to (\ref{eq:5.12}).
From (\ref{eq:5.18}) and (\ref{eq:5.19}), we obtain
\begin{equation}\label{eq:5.19}
\begin{aligned}
&\biggl\{(k+\dfrac1{2n}-\dfrac14)\dfrac{s_1\delta}{k^2(s_1+1)}-1\biggl\}\int_{M}(S+\epsilon)^{2k}Sf^2dv\\
&\leq \bigl\{\dfrac{s_1}{k^2}(k+\dfrac1{2n}-\dfrac14)\int_M(S+\epsilon)^{2k}|\nabla f |^2dv\\
\end{aligned}
\end{equation}
Since $\dfrac{s_1}{(s_1+1)}\to 1$ if $s_1\to \infty$, in the same way, we know
\begin{equation}\label{eq:5.9}
\begin{aligned}
&\int_{M}(S+\epsilon)^{2k}Sf^2dv
\leq C\int_M(S+\epsilon)^{2k}|\nabla f |^2dv\\
\end{aligned}
\end{equation}
for a sufficient large $s_1$.
Letting $\epsilon\to 0$, we conclude
\begin{equation}\label{eq:5.9}
\begin{aligned}
&\int_{M}S^{2k+1}f^2dv
\leq C\int_MS^{2k}|\nabla f |^2dv.\\
\end{aligned}
\end{equation}
Taking $p=4k+2$ and in place of $f$ by $f^{\frac p2}$ and making use of H\"older inequality,
we have
\begin{equation}\label{eq:5.9}
\begin{aligned}
&\int_{M}(\sqrt {S}f)^pdv
\leq \frac{p^2}{4}C\int_MS^{\frac p2-1}f^{2(\frac p2-1)}|\nabla f |^2dv\\
&\leq \frac{p^2}{4}C\big(\int_M(\sqrt S f)^pdv\bigl)^{\frac{ p- 2}{p}}\big(\int_M|\nabla f |^pdv\big)^{\frac2p}.
\end{aligned}
\end{equation}
We obtain
\begin{equation}\label{eq:5.9}
\begin{aligned}
&\int_{M}(\sqrt {S}f)^pdv
\leq C_2\int_M|\nabla f |^pdv.
\end{aligned}
\end{equation}
We finishes the proof of theorem 1.2.
\begin{flushright}
 $\square$
\end{flushright}
\vskip5mm

\noindent
{\it Proof of theorem 1.3}. Since $\delta_1(n)>\delta_0(n)$ and $\delta_1(n)$-stable is $\delta_0(n)$-stable, from
the theorem 4.2, we know that $X:M^n\to \mathbf R^{n+1}$ has the Euclidean volume growth. According to the
corollary 1.1, we obtain that $X:M^n\to \mathbf R^{n+1}$ is the hyperplane.
\begin{flushright}
 $\square$
\end{flushright}
\bibliographystyle{amsplain}

\end {document}